\documentclass[12pt]{amsart}
\usepackage{hyperref,amssymb}

\theoremstyle{plain}
\newtheorem{theorem}{Theorem}
\newtheorem{proposition}[theorem]{Proposition}
\newtheorem{corollary}[theorem]{Corollary}
\newtheorem{lemma}[theorem]{Lemma}

\theoremstyle{definition}

\newcommand{\term}[1]{{\textit{\textbf{#1}}}}   
\renewcommand{\mid}{\::\:}

\DeclareSymbolFont{bbold}{U}{bbold}{m}{n}
\DeclareSymbolFontAlphabet{\mathbbold}{bbold}
\def\one{\mathbbold{1}}

\DeclareMathOperator{\supp}{supp}

\DeclareMathOperator{\Int}{Int}

\renewcommand{\le}{\leqslant}
\renewcommand{\ge}{\geqslant}

\begin{document}

\title{A representation of sup-completion}

\author{Achintya Raya Polavarapu}
\email{polavara@ualberta.ca}

\author{Vladimir G. Troitsky}
\email{troitsky@ualberta.ca}

\address{Department of Mathematical and Statistical Sciences,
  University of Alberta, Edmonton, AB, T6G\,2G1, Canada.}

\thanks{The second author was supported by an NSERC grant.}
\keywords{vector lattice, sup completion, universal completion}
\subjclass[2010]{Primary: 46A40.}


\date{\today}

\begin{abstract}
  It was showed by Donner in~\cite{Donner:82} that every order
  complete vector lattice $X$ may be embedded into a cone~$X^s$,
  called the sup-completion of~$X$. We show that if one represents the
  universal completion of $X$ as $C^\infty(K)$, then $X^s$ is the set
  of all continuous functions from $K$ to $[-\infty,\infty]$ that
  dominate some element of~$X$. This provides a functional
  representation of~$X^s$, as well as an easy alternative proof of its
  existence.
\end{abstract}

\maketitle

\section{Introduction and preliminaries}

The concept of a sup-completion of a vector lattice was introduced
in~\cite{Donner:82} and then further investigated
in~\cite{Azouzi:19,Azouzi:22}; this concept was utilized in a series
of papers by J.Grobler and C.Labuschagne
\cite{Grobler:14,Grobler:17a,Grobler:17b,Grobler:17c,Grobler:19}.  The
intuitive idea behind sup-completion is rather simple: one wants to
enlarge a function space by allowing functions that take value
$+\infty$ on non-negligible sets. In particular, the sup-completion of
$\mathbb R$ is $(-\infty,+\infty]$. However, the formal definition of
sup-completion and the proof of its existence (for order complete
vector lattices) in~\cite{Donner:82} is quite technical. In this note,
we provide an alternative (and, hopefully, more intuitive) way of
constructing sup-completions.

We refer the reader to~\cite{Aliprantis:03,Aliprantis:06} for
background on vector lattices. Throughout this note, all vector
lattices are assumed to be Archimedean. If $\Omega$ and $\Omega'$ are
two Hausdorff topological spaces, we write $C(\Omega,\Omega')$ for the
set of all continuous functions from $\Omega$ to~$\Omega'$.  We write
$\overline{\mathbb R}$ for $[-\infty,\infty]$. For any interval $I$
in~$\overline{\mathbb R}$, we equip $C(\Omega,I)$ with the pointwise
partial order. We write $C(\Omega)$ for $C(\Omega,\mathbb R)$; this is
a vector lattice under pointwise operations. If $K$ is an extremally
disconnected Hausdorff topological space then $C(K)$ is order (or
Dedekind) complete. Suppose that $K$ is an extremally disconnected
compact Hausdorff topological space; we then write $C^\infty(K)$ for
the vector lattice of all continuous functions from $K$ to
$[-\infty,\infty]$ that are finite almost everywhere (a.e.), that is,
except on a nowhere dense set. Scalar multiplication and lattice
operations on $C^\infty(K)$ are defined pointwise. Addition is defined
a.e.: for $f,g\in C^\infty(K)$ there exists a unique
$h\in C^\infty(K)$ such that $f(t)$ and $g(t)$ are finite and
$h(t)=f(t)+g(t)$ for all $t$ in an open dense set; we then define
$h=f+g$. For a function $u\in C(K,\overline{\mathbb R})$ and
$\lambda\in\overline{\mathbb R}$, we write $\{u\le \lambda\}$ as a
shorthand for $\bigl\{t\in K\mid u(t)\le\lambda\bigr\}$; we define
$\{u=\lambda\}$ in a similar fashion.

Maeda-Ogasawara Theorem asserts that every vector lattice $X$ may be
represented as an order dense sublattice of $C^\infty(K)$, where $K$
is an extremally disconnected compact Hausdorff topological space.  If
$X$ is order complete then it is an ideal in $C^\infty(K)$.

Here is the main theorem of this paper:

\begin{theorem}\label{main}
  Let $X$ be an order complete vector lattice represented in $C^\infty(K)$ as
  above. Then the sup-completion of $X$ is
  \begin{math}
    \bigl\{u\in C(K,\overline{\mathbb R})\mid u\ge f
    \mbox{ for some $f$ in }X\bigr\}.
  \end{math}
\end{theorem}

This theorem may be viewed as an alternative definition of a
sup-completion. Before we prove it, we recall the definition of a
sup-completion from~\cite{Donner:82}. By a \term{cone} we mean a
commutative semigroup with zero $(C,+,0)$ equipped with a non-negative
scalar multiplication operation
$(\lambda,a)\in \mathbb R_+\times C\mapsto\lambda a\in C$, which
satisfies the following conditions:
$\lambda(a+b)=\lambda a+\lambda b$, $(\lambda+\mu)a=\lambda a+\mu a$,
$\lambda(\mu a)=(\lambda\mu)a$, $1a=a$ and $0a=0$ for every $a,b\in C$
and $\lambda,\mu\in\mathbb R$. It is easy to see that the set $C_0$ of
all invertible elements in $C$ is a group. For example, if
$C=(-\infty,\infty]$ then $C$ is a cone and $C_0=\mathbb R$ (we take
$0\cdot\infty=0$). The scalar multiplication on
$\mathbb R_+\times C_0$ may be extended to $\mathbb R\times C_0$ via
$(-r)x=-(rx)$ when $r>0$; it is straightforward that $C_0$ is a vector
space over~$\mathbb R$.

We now impose several additional conditions that describe an order on
$C$ and the way $C_0$ ``sits'' in~$C$:
\begin{enumerate}
 \item\label{order} $C$ is equipped with a partial order, such that 
  $a\le b$ implies $a+c\le b+c$ and $\lambda a\le\lambda b$
  for all $a,b,c\in C$ and $\lambda\in\mathbb R_+$;
 \item\label{lattice} $C$ is a lattice under this order;
 \item\label{greatest} $C$ has the greatest element;
 \item\label{ideal} $C_0$ has an ideal property in $C$ in the sense that
  if $x\in C_0$ and $a\in C$ such that $a\le x$ then $a\in C_0$
 \item\label{oc} $C$ is order complete in the sense that every subset $A$ of
  $C$ has supremum; if $A$ is bounded below then $\inf A$ exists;
 \item\label{o-dense} $C_0$ is order dense in $C$ in the following
  sense: $a=\sup\{x\in C_0\mid x\le a\}$ for every $a\in C$;
 \item\label{dist} $a+(x\wedge b)=(a+x)\wedge(a+b)$ whenever $a,b\in C$
  and $x\in C_0$;
 \item\label{sup-sup} for any two non-empty subsets $A$ and $B$ of $C$,
  if $\sup A=\sup B$ and $x\in C_0$ then
  \begin{math}
    \sup\{a\wedge x\mid a\in A\}=\sup\{b\wedge x\mid b\in B\}
  \end{math}
  in~$C$.
\end{enumerate}
It is easy to see that $C_0$ is an order complete vector lattice. We
say that $C$ is a sup-completion of~$C_0$. More precisely, if $X$ an
order complete vector lattice and $X=C_0$ for a cone $C$ satisfying
the properties listed above, we say that $C$ is a
\term{sup-completion} of~$X$. It was proven in~\cite{Donner:82} that
every order complete vector lattice admits a sup-completion;
Theorem~\ref{main} provides an alternative proof of this.

\section{Proof of Theorem~\ref{main}} 

The proof is tedious but
straightforward. Let $C$ be the set in the
theorem:
\begin{equation}\label{C-def}
    C=\bigl\{u\in C(K,\overline{\mathbb R})\mid u\ge f
    \mbox{ for some $f$ in }X\bigr\}.
\end{equation}
By definition, $X$ is a subset of~$C$. It is easy to see that
the set $\{u>-\infty\}$ is open and dense for every $u\in C$.

We will now define operations on $C$.  Non-negative scalar
multiplication on $C$ is defined pointwise; it clearly satisfies
$(\lambda+\mu)u=\lambda u+\mu u$, $\lambda(\mu u)=(\lambda\mu)u$,
$1u=u$, and $0u=0$ for every $u\in C$ and $\lambda,\mu\in\mathbb
R$. On $X$, it agrees with the non-negative scalar multiplication of
$C^\infty(K)$.

Defining addition on $C$ requires some care. We do it similarly to
$C^\infty(K)$. Recall that $K$ is extremally disconnected.

\begin{lemma}\label{Cinfty-ext}
  Suppose that
  $u\colon U\to\overline{\mathbb R}$ is a continuous function on an
  open dense subset $U$ of~$K$. Then $u$ extends uniquely to a
  function in $C(K,\overline{\mathbb R})$.
\end{lemma}

\begin{proof}
  Since $\overline{\mathbb R}$ is topologically and order isomorphic
  to $[-1,1]$ via, say, $\tan\frac{\pi t}{2}$, we may replace
  $\overline{\mathbb R}$ in the statement with $[-1,1]$. 
  So suppose that
  $u\colon U\to[-1,1]$. Let
  \begin{displaymath}
    G=\Bigl\{v\in C\bigl(K,[-1,1]\bigr)
    \mid\forall t\in U\quad v(t)\ge u(t)\Bigr\}.
  \end{displaymath}
  Since $K$ is extremally disconnected, $C(K)$ is order complete. It
  follows from $G\ge -\one$ that $w:=\inf G$ exists in
  $C(K)$. Clearly, $w\in C\bigl(K,[-1,1]\bigr)$.

  We will now show that $w$ extends~$u$.  Fix $t\in U$. Since $K$ is
  totally disconnected, we can find a clopen set $V$ such that
  $t\in V\subseteq U$. Put $v=u\cdot\one_V+\one_{K\setminus V}$; then
  $v\in G$ and, therefore, $w\le v$; it follows that
  $w(t)\le u(t)$. On the other hand, for
  every $v\in G$ we have $v\ge u\cdot\one_V-\one_{K\setminus V}$, so that
  $w\ge u\cdot\one_V-\one_{K\setminus V}$ and, therefore, $w(t)\ge u(t)$.
 
  This proves that $w$ extends $u$. Since $U$ is dense, the extension
  is unique.
\end{proof}

\begin{corollary}\label{sum-dense}
  Let $u_1,u_2\in C(K,\overline{\mathbb R})$ and let
  $U=\{u_1>-\infty\}\cap\{u_2>-\infty\}$. If $U$ is dense than there
  exists a unique $u\in C(K,\overline{\mathbb R})$ such that
  $u(t)=u_1(t)+u_2(t)$ for all $t\in U$.
\end{corollary}

\begin{proof}
  Define $v\colon U\to\mathbb R\cup\{+\infty\}$ via
  $v(t)=u_1(t)+u_2(t)$; $v$ is well defined and continuous. By
  Lemma~\ref{Cinfty-ext}, $v$ extends to a function
  $u\in C(K,\overline{\mathbb R})$. Uniqueness follows from the
  density of~$U$.
\end{proof}

We are now ready to define addition on $C$. Suppose that $u_1,u_2\in
C$. Let $U=\{u_1>-\infty\}\cap\{u_2>-\infty\}$. 
There exist $f_1,f_2\in X$ such that $f_1\le u_1$ and $f_2\le
u_2$. Let $V$ be the set where both $f_1$ and $f_2$ are finite. Then
$V$ dense. It follows from $V\subseteq U$ that $U$ is
dense. Let $u$ be as in Corollary~\ref{sum-dense}. For every $t\in V$,
we have
\begin{displaymath}
  (f_1+f_2)(t)=f_1(t)+f_2(t)\le u_1(t)+u_2(t)=u(t).
\end{displaymath}
Since $V$ is dense, continuity of $f_1+f_2$ and $u$ implies that
$f_1+f_2\le u$ on $K$. Therefore, $u\in C$. Naturally, we define $u=u_1+u_2$.

Clearly, this definition does not depend on the choice of $f_1$ and~$f_2$.
It is straightforward that $u_1+u_2=u_2+u_1$ and that
$\lambda(u_1+u_2)=\lambda u_1+\lambda u_2$ when
$\lambda\in\mathbb R_+$ and $u_1,u_2\in C$. Furthermore, if
$u_1,u_2,u_3\in C$, we have $(u_1+u_2)+u_3=u_1+(u_2+u_3)$ because the
two continuous functions agree with $u_1(t)+u_2(t)+u_3(t)$ for every
$t$ is the dense set where all the three functions are different from
$-\infty$. This shows that $C$ is a cone.

We claim that $C_0=X$. Indeed, if $u\in C_0$ then $u,-u\in C$, hence,
there exist $f,g\in X$ such that $f\le u\le g$. It follows that $u$ is
finite on an open dense set, hence $u\in C^\infty(K)$. Since $X$ is
order complete, it is an ideal in $C^\infty(K)$ and, therefore,
$u\in X$. Conversely, if $f\in X$ then, clearly, $f$ and $-f$ are both
in~$C$, hence $f\in C_0$. Note that the operations of addition and
non-negative scalar multiplication that we defined on $C$ agree with
those of $C^\infty(K)$ on $C\cap C^\infty(K)$. It follows that the
vector space operations induced on $C_0$ by $C$ agree with the
``native'' operations on $C^\infty(K)$.

We define order on $C$ pointwise. It follows from the definition of
$C$ that if $u\in C$ and $v\in C(K,\overline{\mathbb R})$ with
$u\le v$ then $v\in C$. We will now verify conditions
\eqref{order}--\eqref{sup-sup}.

\eqref{order}, \eqref{lattice}, and \eqref{greatest} are
straightforward. It is easy to see that lattice operations on $C$ are
pointwise.

\eqref{ideal} Suppose that $v\le h$ for some $v\in C$ and $h\in
X$. There exists $f\in X$ such that $f\le v\le h$. It follows that
$v\in C^\infty(K)$ and, furthermore, $v\in X$.

Observe that if $f\le u\le g$ for some $f,g\in X$ and
$u\in C(K,\overline{\mathbb R})$ then $u\in X$. Indeed, it follows
from $f\le u$ that $u\in C$; it now follows from \eqref{ideal} that
$u\in X$.

\eqref{oc}
As in the proof of Lemma~\ref{Cinfty-ext}, we observe that
$C(K,\overline{\mathbb R})$ is order complete. Let $A\subseteq C$ with
$A\ne\varnothing$. It follows that $v:=\sup A$
exists in $C(K,\overline{\mathbb R})$. Take any $w\in A$, then $v\ge w\in
C$ implies $v\in C$, hence $v$ is the supremum of $A$ in~$C$.

Now suppose that $u\le A$ for some $u\in C$. Then $v:=\inf A$ exists in
$C(K,\overline{\mathbb R})$; it follows from $u\le v$ that $v\in C$
and, therefore, $v$ is the infimum of $A$ in~$C$.

\eqref{o-dense} Suppose that $u\in C$ and let
$A=\{f\in X\mid f\le u\}$; we need to show that $u=\sup A$. By the
definition of $C$, $A$ is non-empty; fix some $h\in A$. We clearly
have $A\le u$. Suppose $A\le v$ for some $v\in C$; it suffices to show
that $u\le v$. Suppose not. Then there exists $t_0\in K$ with
$v(t_0)<u(t_0)$. Find a clopen neighbourhood $U$ of $t_0$ such that
$v(t)<u(t)$ for all $t\in U$. Since $\one_U$ is in $C^\infty(K)$ and
$X$ is order dense in $C^\infty(K)$, we can find $g\in X$ such that
$0<g\le\one_U$. Then $g(t_1)>0$ for some $t_1\in U$. Let $f$ be a
scalar multiple of $g$ such that $v(t_1)<f(t_1)<u(t_1)$. It follows
from $f\wedge h\le f\wedge u\le f$ that $f\wedge u\in X$ and, therefore,
$f\wedge u\in A$. However, $(f\wedge u)(t_1)>v(t_1)$, which
contradicts $A\le v$.

\eqref{dist} Let $u,v\in C$ and $f\in X$; we need to prove that
$u+(f\wedge v)=(u+f)\wedge(u+v)$. Let $U$ be the set on which $u$,
$v$, and $f$ are all different from~$-\infty$. Then $U$ is open and
dense, and it is straightforward that the functions $u+(f\wedge v)$
and $(u+f)\wedge(u+v)$ agree on~$U$. Since they are continuous, they
are equal on~$K$.

\eqref{sup-sup} Suppose that $\sup A=\sup B$ for two non-empty subsets
$A$ and $B$ of~$C$, and let $f\in X$. It suffices to show that
$\sup(A\wedge f)\le\sup(B\wedge f)$. Suppose not. Then there exists
$u\in A$ such that $u\wedge f\not\le h$, where $h=\sup(B\wedge
f)$. There exists $t\in K$ such that $u(t)\wedge f(t)>h(t)$. Fix
$\lambda\in\mathbb R$ such that  $u(t)\wedge f(t)>\lambda>h(t)$. By
continuity, we can find a clopen neighbourhood $U$ of $t$ such that for all
$s\in U$ we have
\begin{math}
  u(s)\wedge f(s)>\lambda>h(s)\ge v(s)\wedge f(s)
\end{math}
for $v\in B$. It follows from 
\begin{math}
  u(s)\wedge f(s)>\lambda> v(s)\wedge f(s)
\end{math}
that $u(s)>\lambda>v(s)$ for all $s\in U$ and $v\in B$.

Consider the function $w\in C(K,\overline{\mathbb R})$ that is equal
to $\lambda$ on $U$ and $+\infty$ on~$K\setminus U$. Then $w\ge v$ for
all $v\in B$ and, therefore, $w\ge b$, where $b=\sup B$. It follows
that $u(s)>w(s)\ge b(s)$ for all $s\in U$. This contradicts $\sup A=\sup B$.

This completes the proof of the theorem. We now establish a 
useful distributive property of~$C$. Here $C$ is as in~\eqref{C-def}.

\begin{lemma}\label{distr}
  For every $u\in C$ and $A\subseteq X$, $\sup(u+A)=u+\sup A$.
\end{lemma}

\begin{proof}
  For every $f\in A$ we have $f\le \sup A$, so that $u+f\le u+\sup A$
  and, therefore, $\sup(u+A)\le u+\sup A$.

  To prove the converse inequality, let $U=\{u<\infty\}$; observe that
  $\overline{U}$ is clopen. Viewing $u$, $\sup A$, and $\sup(u+A)$ as
  functions in $C(K,\overline{\mathbb R})$, it suffices to prove the
  inequality for their band projections onto the bands generated by
  $\overline{U}$ and $\overline{U}^C$. On the latter set, $u$ is
  identically~$\infty$, so the inequality is easily satisfied. So
  replacing $K$ with~$\overline{U}$, we may assume that $u$ is a.e.\
  finite and, therefore, $u\in C^\infty(K)$. Then $-u$ is defined in
  $C^\infty(K)$ and we have
  \begin{math}
    \sup A=\sup(-u+u+A)\le -u+\sup(u+A),
  \end{math}
  so that $u+\sup A\le\sup(u+A)$.
\end{proof}

\section{Uniqueness}

Donner in~\cite{Donner:82} proved that the sup-completion is
unique. His proof relies on his construction of a sup-completion. We
now present an alternative proof using our Theorem~\ref{main} instead
(but our proof is built on the same ideas as that
in~\cite{Donner:82}). We need the following variant of Riesz
Decomposition Property:

\begin{lemma}\label{RDP}
  Let $C$ be a sup-completion of~$X$. Suppose that $x\le u+v$ for
  some $x\in X$ and $u,v\in C$ with $v\ge 0$. Then there exist
  $y,z\in X$ such that $x=y+z$, $y\le u$, and $z\le v$.
\end{lemma}

\begin{proof}
  Since $x\wedge u\in X$ by~\eqref{ideal}, we can define $y=x\wedge u$
  and $z=x-x\wedge u$ in $X$. We clearly have $x=y+z$ and $y\le u$. It
  is left to verify that $z\le v$. Since $z\in X$, $v-z=v+(-z)$ is
  defined. Using~\eqref{dist}, we get
  \begin{displaymath}
    v-z=(v-x)+x\wedge u=(v-x+x)\wedge(v-x+u)=v\wedge(u+v-x)\ge 0.
  \end{displaymath}
\end{proof}

\begin{theorem}
  Sup-completion of an order complete vector lattice is unique. That
  is, if $C$ and $D$ are two sup-completions of $X$ then there exists
  a bijection $J\colon D\to C$ such that $u\le
  v$ iff $Ju\le Jv$, $J(\alpha u)=\alpha Ju$, and $J(u+v)=Ju+Jv$ when
  $u,v\in D$ and $\alpha\ge 0$, and $J$ agrees with the identity on~$X$.
\end{theorem}

\begin{proof}
  WLOG, $C$ is the sup-completion that we constructed in
  Theorem~\ref{main}, i.e., $C$ is as in~\eqref{C-def}. For a
  non-empty set $A\subseteq X$, we define
  $J\bigl(\sup_DA)=\sup_CA$. Let's verify that $J$ is
  well-defined. Suppose that $A$ and $B$ are two non-empty subsets of
  $X$ with $\sup_DA=\sup_DB$. For every $a\in A$, \eqref{sup-sup} yields
  \begin{displaymath}
    \textstyle a=\sup_D(A\wedge a)=\sup_D(B\wedge a)
    =\sup_X(B\wedge a)\le\sup_CB.
  \end{displaymath}
  It follows that $\sup_CA\le\sup_CB$. The opposite inequality is
  similar.  It is left to verify that $J$ is defined on all of $D$: if
  $u\in D$ then it follows from~\eqref{o-dense} that $u=\sup_DA$ where
  $A=\{x\in X\mid x\le u\}$. Hence, $J(u)=\sup_CA$, that is,
  \begin{equation}
    \label{Ju}
    J(u)=\textstyle\sup_C\{x\in X\mid x\le u\}
  \end{equation}
  Since the definition of $J$ is symmetric, it is easy to see that $J$
  is a bijection, with the inverse given by
  $J^{-1}\bigl(\sup_CA)=\sup_DA$ for $A\subseteq X$.
  It is straightforward that $J(\alpha u)=\alpha J(u)$ when
  $u\in D$ and $\alpha\ge 0$. It follows from~\eqref{Ju} that $u\le v$
  implies $J(u)\le J(v)$. Since the definition of $J$ is symmetric,
  the converse is also satisfied, hence $u\le v$ iff $J(u)\le
  J(v)$. 
  
  It is left to show that $J$ is additive. If $u\in D$ and $y\in X$,
  it follows from~\eqref{Ju} and Lemma~\ref{distr} that
  \begin{displaymath}
    J(y+u)=\textstyle\sup_C\bigl\{x\in X\mid x\le y+u\bigr\}
    =\sup_C\bigl(y+A)
    =y+\sup_CA=y+J(u),
  \end{displaymath}
  where $A=\{x\in X\mid x\le u\}$.
  
  Now fix $u,v\in D$. Put
  $A=\{x\in X\mid x\le u\}$ and $B=\{y\in X\mid y\le v\}$. Then
  $u=\sup_DA$ and $v=\sup_DB$ by~\eqref{o-dense}, and
  $J(u)=\sup_CA$ and $J(v)=\sup_CB$ by~\eqref{Ju}. Fix $x\in A$
  and $y\in B$. We have $x+y\le u+v$. It follows from~\eqref{Ju} that
  $x+y\le J(u+v)$, so that $x\le J(u+v)-y$. Taking supremum in $C$ over $x\in
  A$, we get $J(u)\le J(u+v)-y$ and, therefore, $J(u)+y\le J(u+v)$. By
  Lemma~\ref{distr}, we have
  \begin{displaymath}
    J(u)+J(v)=\textstyle J(u)+\sup_CB=\sup_C(J(u)+B)=\sup_C\{J(u)+y\mid y\in B\}
    \le J(u+v).
  \end{displaymath}

  To prove the other inequality, we first assume that $v\ge 0$. Note
  that $J(u+v)=\sup_CB$, where $B=\{x\in X\mid x\le u+v\}$. For every
  $x\in B$, find $y$ and $z$ as in Lemma~\ref{RDP}. It follows that
  \begin{displaymath}
    B\subseteq\{y\in X\mid y\le u\}+\{z\in X\mid z\le v\}\le J(u)+J(v),
  \end{displaymath}
  so that $J(u+v)\le J(u)+J(v)$.

  Finally, if $u$ and $v$ are two arbitrary elements of $D$, it
  follows from $v^-\in X$ that
  \begin{math}
    J(u+v)=J(u+v^+-v^-)=J(u+v^+)-v^-\le J(u)+J(v^+)-v^-
    =J(u)+J(v^+-v^-)=J(u)+J(v).
  \end{math}
\end{proof}

\section{Applications}

We now use our representation of sup-completion to provide simple
proofs of several results of~\cite{Azouzi:19,Azouzi:22} and improve
some of them. Throughout this section, $X$ is (again) an order
complete vector lattice. We write $X^u$ and $X^s$ for the universal
completion and the sup-completion of $X$, respectively. As before, we
represent $X^u$ as an order dense sublattice of $C^\infty(K)$ for some
extremally disconnected compact $K$; we represent $X^s$ as in
Theorem~\ref{main}.

We start by revisiting Corollary~7 in~\cite{Azouzi:19}. It follows
immediately from Theorem~\ref{main} that every non-negative function
in $C(K,\overline{\mathbb R})$ belongs to~$X^s$. In particular, we
have $X^u_+\subseteq X^s$. Suppose now that $Y$ is an order dense
order complete sublattice of~$X$. Then clearly $Y$ is still order
dense in $C^\infty(K)$, hence $C^\infty(K)=Y^u$, and
\begin{displaymath}
    Y^s=\bigl\{u\in C(K,\overline{\mathbb R})\mid u\ge f
    \mbox{ for some $f$ in }Y\bigr\}.
\end{displaymath}
It follows that both $(X^s)_+$ and $(Y^s)_+$ consist of all
non-negative function in $C(K,\overline{\mathbb R})$ and, therefore,
$(X^s)_+=(Y^s)_+$. Note that if $Y^s=X^s$ then $Y=X$ because every
negative $f\in X$ belongs to $Y$ by~\eqref{ideal}.

Recall that if $X$ has a weak unit $e$, one can choose the
representation so that $e=\one$.

\begin{proposition}[\cite{Azouzi:19}]
  If $e$ is a weak unit in $X_+$ and $0\le u\in X^s$ then
  $u=\sup_n(ne\wedge u)$.
\end{proposition}

\begin{proof}
  WLOG, $e=\one$. Let $v=\sup_n(n\one\wedge u)$. Clearly, $v\le
  u$. Fix $t\in K$. Then $v(t)\ge n\wedge u(t)$ for all
  $n$. Considering separately the cases when $u(t)=\infty$ and when
  $u(t)<\infty$, we see that $v(t)\ge u(t)$ and, therefore, $v\ge u$.
\end{proof}

By Maeda-Ogasawara theory (see, e.g., Chapter~7
in~\cite{Aliprantis:03}), there is a one-to-one correspondence between
clopen subsets of $K$ and bands in~$X$: if $U$ be a clopen set in $K$
then the set $\{x\in X\mid\supp x\subseteq U\}$ is a band in~$X$, and
every band in $X$ is of this form; we denote it by~$B_U$.  The
corresponding band projection $P_U$ is given by
$P_Ux=x\cdot\one_{U}$. It is clear that the universal completion of
$B_U$ is $C^\infty(U)$ and, hence the sup-completion of $B_U$ can be
computed as in Theorem~\ref{main}.

In particular, for $a\in X$, the principal
band projection $P_a$ is given by the following: for $x\in X$ and
$t\in K$, we have
\begin{displaymath}
  (P_ax)(t)=
  \begin{cases}
    x(t) & \mbox{ if $t\in U$, and}\\
    0 & \mbox{ otherwise },
  \end{cases}
  \qquad\mbox{ where }U=\overline{\{a\ne 0\}}.
\end{displaymath}
The preceding formula clearly extends to the case when $a\in X^s$,
yielding a band projection on $X$.

\smallskip

In Theorem~15 of~\cite{Azouzi:22}, the authors prove that every
element $u$ in $X^s$ can be split into its finite and infinite
parts. Using our representation, this is now easy: let $U$ be the
closure of $\{u<\infty\}$, then the finite part $x$ of $u$ is defined
as $x=P_Uu$ and is the function that agrees with $u$ on $U$ and
vanishes on $U^C$, while the infinite part $w$ of $u$ is defined as
$P_{U^C}u$ and is the function that vanishes on $U$ and is identically
$\infty$ on $U^C$. Clearly, $u=x+w$ and $x\perp w$. It follows from
$x\in C^\infty(K)$ that $x\in X^u$. It follows from $u\in X^s$ that
$u\ge f$ for some $f\in X$, hence $P_Uu\ge P_Uf$ and, therefore,
$x\in X^s$. Clearly, $u\in X^u$ iff its infinite part $w$ equals zero
and $u$ equals $x$, its finite part. Note also that $x$ may be viewed
as the restriction of $u$ to $U$; so we may view $x$ as an element of
$C^\infty(U)$ and of~$U^s$.

We now extend the Riesz Decomposition Theorem to $X^s$;
cf.~Lemma~\ref{RDP} and \cite[Lemma~1]{Azouzi:19}.

\begin{lemma}\label{RDPs}
  Suppose that $x\le u+v$ for
  some $x,u,v\in X^s$ with $v\ge 0$. Then there exist
  $y,z\in X^s$ such that $x=y+z$, $y\le u$, and $z\le v$.
\end{lemma}

\begin{proof}
  Let $U$ be the closure of $\{u<\infty\}$. Then $U$ is clopen and
  $X=B_U\oplus B_{U^C}$.  It suffices to prove the statement on $U$
  and on~$U^C$. That is, replacing $K$ with $U$ or~$U^C$, we reduce
  the problem to two special cases: when $u$ is constant infinity and
  when $u<\infty$ on a dense open set. If $u$ is constant infinity
  then we take $y=x$ and $z=0$; it is clear that all the requirements
  are satisfied. Suppose now that $u<\infty$ on a dense open set. That
  is, $u\in C^\infty(K)$.

  Put $y=x\wedge u$. Then $y\in X^s$ by~\eqref{lattice}, and
  $y\le u$ yields $y\in C^\infty(K)$. It follows that
  $z:=x-x\wedge u=x-y$ is defined in $X^s$. Clearly, $x=y+z$
  in~$X^s$. It is left to show that $z\le v$.

  Again, this is trivially satisfied on the band corresponding to the
  infinite part of $v$ as $v$ is constant infinity there. So, passing
  to the complementary band corresponding to
  $\overline{\{v<\infty\}}$, we may assume WLOG that $v\in
  C^\infty(K)$. It follows from $x\le u+v$ that $x\in C^\infty(K)$,
  hence also $y,z\in C^\infty(K)$. As in the proof of  Lemma~\ref{RDP}, we have
  \begin{displaymath}
    v-z=(v-x)+x\wedge u=(v-x+x)\wedge(v-x+u)=v\wedge(u+v-x)\ge 0.
  \end{displaymath}
\end{proof}

\begin{theorem}[\cite{Azouzi:19}]\label{P-lambda}
  Let $e\in X_+$ be a weak unit and $u\in X^s$. Then 
  \begin{displaymath}
    u\notin X^u\qquad\mbox{ iff }\qquad
    \inf_{\lambda\in(0,\infty)}P_{(u-\lambda e)^+}e>0.
  \end{displaymath}
\end{theorem}

\begin{proof}
  WLOG, we may chose the representation so that $e=\one$. Let
  $v=\inf_{\lambda\in(0,\infty)}v_\lambda$, where
  $v_\lambda=P_{(u-\lambda\one)^+}\one$.  Suppose that $u\notin X^u$. Then
  $\Int\{u=\infty\}$ in non-empty; denote this set by~$V$. For every
  $t\in V$ and every $\lambda\in(0,\infty)$ we have
  $(u-\lambda\one)^+(t)=\infty$, so that $v_\lambda(t)=1$. It follows
  that $v_\lambda\ge\one_{\overline{V}}$ and, therefore,
  $v\ge\one_{\overline{V}}>0$.

  Conversely, suppose that $v>0$. Then the set $W:=\{v>0\}$ is
  open. Fix $t\in W$. For every $\lambda\in(0,\infty)$ it follows from
  $v_\lambda\ge v$ that $v_\lambda(t)>0$, so that
  $(u-\lambda\one)^+(t)>0$ and, therefore, $u(t)\ge\lambda$. It
  follows that $u(t)=\infty$ for all $t\in W$ and therefore, $u\notin
  C^\infty(K)$.
\end{proof}

It is proved in Theorem~19 in~\cite{Azouzi:22} that if $0\le u\in X^s$
and $e$ is a weak unit in $X$ then
$P_we=\inf_{\lambda\in(0,\infty)}P_{(u-\lambda e)^+}e$, where $w$ is
the infinite part of $u$. This fact can now be easily proved
analogously to Theorem~\ref{P-lambda}. Other properties of
decompositions of $u$ into the finite and the infinite part
in~\cite{Azouzi:22} can be proved in a similar way.

\smallskip

Fix a weak unit $e$ in~$X_+$. WLOG, we may assume that $e$ corresponds
to $\one$ in the $C^\infty(K)$ representation of~$X^u$. Similarly to
how we defined addition on $C^\infty(K)$, one can define
multiplication, making $X^u$ into an f-algebra with $e$ being a
multiplicative unit. Recall that $(X^s)_+$ consists of all continuous
positive functions from $K$ to $[0,\infty]$. Similarly to how we
defined addition on $X^s$, we can define product on $(X^s)_+$. That
is, $uv=w$ iff $u(t)v(t)=w(t)$ for all $t$ in an open dense set; we
again follow the convention that $0\cdot\infty=0$. It is easy to see
that the resulting product agrees with that defined in Section~3.2
of~\cite{Azouzi:22}.

\end{document}